\documentclass[11pt]{amsart}

\usepackage{amsmath}
\usepackage{amssymb}
\usepackage{amsthm}
\usepackage{mathrsfs}
\usepackage{enumitem}
\usepackage[all]{xypic}
\usepackage[bookmarks=true,bookmarksnumbered=true]{hyperref}
\usepackage[hmargin=2cm,vmargin=3cm]{geometry}

\numberwithin{equation}{section}

\theoremstyle{plain}
\newtheorem{thm}{Theorem}[section]

\newtheorem{prop}[thm]{Proposition}

\newtheorem*{main}{Main Theorem}
\newtheorem*{B}{Theorem (B)$_n$}

\theoremstyle{definition}

\theoremstyle{remark}
\newtheorem{rem}[thm]{Remark}

\def\disc{\operatorname{disc}}

\def\Hom{\operatorname{Hom}}

\def\Irr{\operatorname{Irr}}

\def\N{\operatorname{N}}

\def\Tr{\operatorname{Tr}}

\def\GL{\mathrm{GL}}

\def\SL{\mathrm{SL}}

\def\U{\mathrm{U}}

\def\WD{\mathit{WD}}

\def\CC{\mathbb{C}}

\def\ZZ{\mathbb{Z}}

\def\1{\Eins}

\makeatletter
\def\varddots{\mathinner{\mkern1mu
    \raise\p@\hbox{.}\mkern2mu\raise4\p@\hbox{.}\mkern2mu
    \raise7\p@\vbox{\kern7\p@\hbox{.}}\mkern1mu}}
\makeatother

\newcommand{\BIGOP}[1]{\mathop{\mathchoice%
{\raise-0.22em\hbox{\huge $#1$}}%
{\raise-0.05em\hbox{\Large $#1$}}{\hbox{\large $#1$}}{#1}}}

\newcommand{\BIGboxplus}{\mathop{\mathchoice%
{\raise-0.35em\hbox{\huge $\boxplus$}}%
{\raise-0.15em\hbox{\Large $\boxplus$}}{\hbox{\large $\boxplus$}}{\boxplus}}}

\title{On the Fourier-Jacobi model for some endoscopic Arthur packet of $U(3) \times U(3)$ : the non-generic case}
\author{Jaeho Haan}
\address{Algebraic Structure and its Applications Research Center(ASARC), Department
of Mathematics, Korea Advanced Institute of Science and Technology}
\email{jaehohaan@gmail.com}

\date{\today}

\begin{document}

\begin{abstract}For a generic $L$-parameter of $U(n)\times U(n)$, it is conjectured that there is a unique representation in their associated relevant Vogan $L$-packet which produces the unique Fourier-Jacobi model. We investigated this conjecture for some non-generic $L$-parameters of $U(3)\times U(3)$ and discovered that this conjecture is still true for some non-generic $L$-parameter and false for some non-generic $L$-parameter. In the case when it holds, we specified such representation under the local Langlands correspondence for unitary group.\\

\end{abstract}
\maketitle
\section{\textbf{Introduction}}The local Gan-Gross-Prasad conjecture deals with certain restriction problems between $p$-adic groups. In this paper, we shall investigate it for some non-generic case not yet treated before. 

Let $E/F$ be a quadratic extension of number fields and $G=U(3)$ be the quasi-split unitary group of rank 3 relative to  $E/F$. Then $H=U(2) \times U(1)$ is the unique elliptic endoscopic group for $G$. 
In \cite{Ro}, Rogawski has defined a certain enlarged class of $L$-packets, or $A$-packets, of $G$ using endoscopic transfer of one-dimensional characters of $H$ to $G$. In more detail, let $\varrho= \otimes_{v} \varrho_v$ be a one-dimensional automorphic character of $H$. The $A$-packet $\Pi(\varrho) \simeq \otimes \Pi(\varrho_v)$ is the transfer of $\varrho$ with respect to functoriality for an embedding of $L$-groups $\xi : ^LH \to ^LG$. Then for all places $v$ of $F$, $\Pi(\varrho_v)$ contains a certain non-tempered representation $\pi^{n}(\varrho_v)$ and it contains an additional supercuspidal representaton $\pi^{s}(\varrho_v)$ precisely when $v$ remains prime in $E$. Gelbart and Rogawski \cite{Ge1} showed that the representations in this $A$-packet arise in the Weil representation of $G$. Our goal is to study the branching rule of the representations in this $A$-packet.

For the branching problem, there is a fascinating conjecture, the so-called Gan-Gross-Prasad (GGP) conjecture, which was first proposed by Gross and Prasad \cite{ggp} for orthogonal group and later they, together with Gan, extended it to all classical group in \cite{Gan2}. Since our main theorem has to do with it, we shall give a brief review on the GGP conjecture, especially for unitary group.

Let $E/F$ be a quadratic extension of non-archimedean local fields of characteristic zero. Let $V_{n+1}$ be a Hermitian space of dimension $n+1$ over $E$ and $W_n$ a skew-Hermitian space of dimension $n$ over $E$.
Let $V_n \subset V_{n+1}$  be a nondegenerate subspace of codimension $1$,  so that if we set 
\[  G_n =  \U(V_n) \times \U(V_{n+1}) \quad \text{or} \quad \U(W_n) \times \U(W_n) \]
and
\[   H_n = \U(V_n) \quad \text{or} \quad \U(W_n), \]
then we have a diagonal embedding
\[ \Delta:  H_n \hookrightarrow G_n. \]

Let $\pi$ be an irreducible smooth representation of $G_n$. In the Hermitian case, one is interested in computing 
\[  \dim_\CC \Hom_{\Delta H_n} ( \pi, \CC). \]
We shall call this the \emph{Bessel} case (B) of the GGP conjecture.
For the GGP conjecture in the skew-Hermitian case, we need to introduce a certain Weil representation $\omega_{\psi, \chi, W_n}$ of $H_n$, where $\psi$ is a nontrivial additive character of $F$ and $\chi$ is a character of $E^{\times}$ whose restriction to $F^{\times}$ is the quadratic character $\omega_{E/F}$ associated to $E/F$ by local class field theory. (For the exact definition of $\omega_{\psi, \chi, W_n}$, please refer to Section. \ref{SS}.) In this case, one is interested in computing
\[  \dim_\CC \Hom_{\Delta H_n} ( \pi, \omega_{\psi,\chi, W_n}). \]
We shall call this the \emph{Fourier--Jacobi} case (FJ) of the GGP conjecture. To treat both cases using one notation, we shall let $\nu = \CC$ or $\omega_{\psi,\chi, W_n}$ in the respective cases.

By the results of \cite{agrs} \cite{sun}, it is known that \[  \dim_\CC \Hom_{\Delta H_n} ( \pi, \nu ) \le 1 \] and so the next step is to specify irreducible smooth representations $\pi$ such that $$\Hom_{\Delta H_n} ( \pi, \nu ) = 1. $$(A non-zero element of $\Hom_{\Delta H_n} ( \pi, \nu )$ is called a \emph{Bessel} (\emph{Fourier-Jacobi}) model of $\pi$ in the (skew) hermitian case.)

In \cite{Gan2}, Gan, Gross, Prasad has brought this problem into a more general setting using the notion of relevant pure inner forms of $G_n$ and Vogan $L$-packets. A pure inner form of $G_n$ is a group of the form 
\[  G_n' = \U(V_n') \times \U(V'_{n+1}) \quad \text{or} \quad \U(W'_n) \times \U(W''_n) \]
where $V_{n}' \subset V_{n+1}'$ are $n$ and $n+1$ dimensional hermitian spaces over $E$ and $W_n', W_n''$ are $n$-dimensional skew hermitian spaces over $E$.

\noindent Furthermore, if \[  \quad  V_{n+1}'/V_{n}' \cong V_{n+1}/V_{n} \quad \text{or} \quad W_n'=W_n'', \] we say that $G_n'$ is relevant pure inner form.\\ (Indeed, there are four pure inner forms of $G_n$ and among them, only two are relevant.)

If $G_n'$ is relevant, we set
\[   H'_n = \U(V'_n) \quad \text{or} \quad \U(W'_n), \]
so that we have a diagonal embedding
\[ \Delta:  H'_n \hookrightarrow G'_n. \]

Now suppose that $\phi$ is an $L$-parameter for the group $G_n$. Then the (relevant) Vogan $L$-packet $\Pi_{\phi}$ associated to $\phi$  consists of certain irreducible smooth representations of $G_n$ and its (relevant) pure inner forms $G_n'$ whose $L$-parameter is $\phi$. We denote the relevant Vogan $L$-packet of $\phi$ by $\Pi^R_{\phi}$.

With these notions, we can loosely state the result of Beuzart-Plessis (\cite{bp1},\cite{bp2},\cite{bp3}) for \emph{Bessel} case and Gan-Ichino (\cite{iw}) for \emph{Fourier-Jacobi} case as follows :
\begin{thm}
For a tempered $L$-parameter $\phi$ of $G_n$, the followings hold:
\begin{enumerate}
\item $\sum_{\pi' \in \Pi^R_{\phi}}\dim_{\CC}\Hom_{\Delta H'_n} ( \pi', \nu )=1.$

\item Using the local Langlands correspondence for unitary group, we can pinpoint the unique $\pi' \in \Pi^R_{\phi}$ such that $$\dim_{\CC}\Hom_{\Delta H'_n} ( \pi', \nu )=1.$$ 
\end{enumerate}
\end{thm}
To emphasize its dependence on the number $n$, we denote the $Bessel$ and $\emph{Fourier-Jacobi}$ case of Theorem 1.1 as (B)$_n$ and (FJ)$_n$ respectively, and later we shall elaborate more on this notation. \\The GGP conjecture predicts that this theorem also holds for a generic $L$-parameter $\phi$ of $G_n$.

Our main theorem is to investigate (FJ)$_3$ for some $L$-parameter of $G_3$ involving a non-generic $L$-parameter of $U(W_3)$. More precisely, we have

\begin{main}For an irreducible smooth representation $\pi_2$ of $U(W_3)$, let $\pi=\pi^{n}(\varrho) \otimes \pi_2$ as a representation of $G_3$. Then

\begin{enumerate}
\item $\Hom_{\Delta H_3}(\pi,\omega_{\psi,\chi,W_3})=0$ if $\pi_2$ is not a theta lift from $U(V_2)$.
\item Assume that $\pi_2$ is the theta lift from $U(V_1')$ and let $\phi=\phi^{n}\otimes \phi_2$ be the $L$-parameter of $\pi$. Then $$\sum_{\pi' \in \Pi^R_{\phi}}\dim_{\CC}\Hom_{\Delta H'_3} ( \pi', \omega_{\psi,\chi,W_3} )=1.$$ 
\item Using the local Langlands correspondence for unitary groups, we can explicitly describe the representation $\pi' \in \Pi^R_{\phi}$ appearing in \text{(ii)} such that $$\label{p}\dim_{\CC}\Hom_{\Delta H'_3} ( \pi', \omega_{\psi,\chi,W_3} )=1.$$
\end{enumerate}

\end{main}

\begin{rem} As we shall see in Theorem \ref{Te}, the $L$-parameter of $\pi^{n}(\varrho_v)$ is not only non-tempered but also non-generic. Thus if we choose the $L$-parameter of $\phi_2$ in $\phi$ apart from those obtained by the theta lift from $U(V_2)$ to $U(W_3)$, then the first part of the Main Theorem tells us that the GGP conjecture may not  be true for non-generic $L$-parameter of $G_n$.

\end{rem}The proof of Main Theorem is based on the following see-saw diagram :
\[
 \xymatrix{
  \U(W_3)  \times \U(W_3)  \ar@{-}[dr] \ar@{-}[d] & \U(V_{2}) \ar@{-}[d] \\
  \U(W_3) \ar@{-}[ur] &  \U(V_1) \times \U(V_{1})}.
\]

Since all elements in the $A$-packet $\Pi(\varrho)$ can be obtained by theta lift from $U(V_1)$, we can write $\pi^n(\varrho)=\Theta_{\psi,\chi,W_3,V_1}(\sigma)$ where $\sigma$ is an irreducible smooth character of $U(V_1)$ and $\psi,\chi$ are some characters, which are needed to fix a relevant Weil representation. Then by the see-saw identity, we have 
$$\Hom_{U(W_3)}(\Theta_{\psi,\chi,W_3,V_1}(\sigma)\otimes  \omega^{\vee}_{\psi,\chi,W_3},\pi^{\vee}_2) \simeq \Hom_{U(V_1)}(\Theta_{\psi,\chi,V_2,W_3}(\pi^{\vee}_2),\sigma).$$ From this, we see that for having $\Hom_{U(W_3)}(\Theta_{\psi,\chi,W_3,V_1}(\sigma)\otimes  \omega^{\vee}_{\psi,\chi,W_3},\pi^{\vee}_2) \ne0$, it should be preceded $\Theta_{\psi,\chi,V_2,W_3}(\pi^{\vee}_2)\ne 0$. This accounts for (i) in the \textbf{Main Theorem} because $$\Hom_{U(W_3)}(\Theta_{\psi,\chi,W_3,V_1}(\sigma)\otimes  \omega^{\vee}_{\psi,\chi,W_3},\pi^{\vee}_2)\simeq \Hom_{U(W_3)}(\Theta_{\psi,\chi,W_3,V_1}(\sigma)\otimes \pi_2, \omega_{\psi,\chi,W_3}).$$

If $\Theta_{\psi,\chi,V_2,W_3}(\pi^{\vee}_2)\ne 0$, then by the local theta correspondence, $\pi^{\vee}_2$ should be $\Theta_{\psi,\chi,W_3,V_2}(\pi_0)$, where $\pi_0$ is an irreducible representation of $U(V_2)$. By applying (B)$_1$, we can pinpoint $\pi_0$ and $\sigma$ in the framework of local Langlands correspondence such that $\Hom_{U(V_1)}(\pi_0,\sigma)\ne0$. Next we shall use the precise local theta correspondences for $\big(U(V_1),U(W_3)\big)$ and $\big(U(V_1),U(W_1)\big)$ in order to transfer recipe for (B)$_1$ to (FJ)$_3$. 

The rest of the paper is organized as follows; In Section 2, we shall give a brief sketch of the local Langlands correspondence for unitary group. In Section 3, we collect some results on the local theta correspondence for unitary group which we will use in the proof of our main results. In Section 4, we shall prove our Main Theorem.

\subsection*{Acknowledgements} The author would like to express deep gratitude to professor Haseo Ki for introducing him this research area. The influence of the work of Gan and Ichino \cite{iw} on this paper is obvious to the reader. We are grateful to professor Atsushi Ichino for his kind suggestion to pursue the non-tempered aspects of the GGP conjecture. The author expresses his deepest thanks to the referee for his many invaluable comments and bringing our attention to the work of H.Atobe and Gan \cite{ag} containing Theorem 4.1 and Proposition 5.4, both of which are crucial to prove our main theorem. This work was supported by the National Research Foundation of Korea(NRF) grant funded by the Korea government(MSIP)(ASARC, NRF-2007-0056093).

\subsection{Notations}\label{not}We fix some notations we shall use throughout this paper:
\begin{itemize}
\item  $E/F$ is a quadratic extension of non-archimedean local fields of characteristic zero.

\item  $c$ is the non-trivial element of Gal$(E/F)$.

\item $\text{Fr}_E$ is a Frobenius element of Gal$(\bar{E}/E)$.

\item  Denote by $\Tr_{E/F}$ and $\N_{E/F}$ the trace and norm maps from $E$ to $F$.

\item $\delta$ is an element of $E$ such that  $\Tr_{E/F}(\delta)=0$.

\item Let $\psi$ be an additive character of $F$ and define $$\psi^E(x)  := \psi(\frac{1}{2}\Tr_{E/F}(\delta x)) \quad \text{and} \quad \psi^E_2(x)  := \psi(\Tr_{E/F}(\delta x)).$$

\item  Let $\chi$ be a character of $E^{\times}$ whose resriction to $F^{\times}$ is $\omega_{E/F}$, which is the quadratic character assosiated to $E/F$ by local class field theory. 

\item For a linear algebraic group $G$, its $F$-points will be denoted by $G(F)$ or simply by $G$.

\end{itemize}

\section{\textbf{Local Langlands correspondence for unitary group}}
By the recent work of Mok \cite{Mok} and Kaletha-M\'inguez-Shin-White \cite{kmsw}, the local Langlands correpondence is now known for unitary group conditional on the stabilization of the twisted trace formula and weighted fundamental lemma. The twisted trace formula has now been stabilized by Waldspurger \cite{Stab1}--\cite{Stab9} and Moeglin-Waldspurger \cite{Stab6}--\cite{Stab10} and the proof of the weighted fundamental lemma is an ongoing project of Chaudouard and Laumon.   Since our main results are expressed using the local Langlands correspondence, we shall assume the local Langlands correspondence for unitary group. In this section, we list some of its properties which are used in this paper. Indeed, much of this section are excerpts from Section.2 in \cite{iw}.

\subsection{Hermitian and skew-Hermitian spaces}
For $\varepsilon=\pm1$, let $V$ be a finite $n$-dimensional vector space over $E$ equipped with a nondegenerate $\varepsilon$-hermitian $c$-sesquilinear form $\langle \cdot, \cdot \rangle_V : V \times V \rightarrow E$.
That means for $v, w \in V$ and $a, b \in E$,  \[
 \langle a v, b w \rangle_V = a b^c \langle v, w \rangle_V, \qquad
 \langle w, v \rangle_V = \varepsilon \cdot \langle v, w \rangle_V^c.
\]

\noindent We define $\disc V = (-1)^{(n-1)n/2} \cdot \det V$, so that 
\[
 \disc V \in
 \begin{cases}
  F^{\times} / \mathrm{N}_{E/F}(E^{\times})
  & \text{if $\varepsilon = +1$;} \\
  \delta^n \cdot F^{\times} / \mathrm{N}_{E/F}(E^{\times})
  & \text{if $\varepsilon = -1$}
 \end{cases}
\]
and we can define $\epsilon(V) = \pm 1$ by 
\begin{equation}\label{sign}
 \epsilon(V) = 
 \begin{cases}
  \omega_{E/F}(\disc V) & \text{if $\varepsilon = +1$;} \\
  \omega_{E/F}(\delta^{-n} \cdot \disc V) & \text{if $\varepsilon = -1$.}
 \end{cases}
\end{equation}
By a theorem of Landherr, for a given positive integer $n$, there are exactly two isomorphism classes of $\varepsilon$-hermitian spaces of dimension $n$ and they are distinguished from each other by $\epsilon(V)$.\\
Let $\U(V)$ be the unitary group of $V$ defined by
\[
  \U(V) = \{ g \in \GL(V) \, | \,
 \text{$\langle g v, g w \rangle_V =  \langle v, w \rangle_V$ for $v, w \in V$}
 \}.
\]
Then $U(V)$ turns out to be connected reductive algebraic group defined over $F$.
\subsection{$L$-parameters and component groups}

Let $I_F$ be the inertia subgroup of $\text{Gal}(\bar{F}/F)$. Let $W_F=I_F \rtimes \langle \text{Fr}_F \rangle $ be the Weil group of $F$ and $\WD_F = W_F \times \SL_2(\CC)$ the Weil-Deligne group of $F$.
For a homomorphism $\phi: \WD_F \rightarrow \GL_n(\CC)$, we say that it is a representation of $\WD_F$ if 
\begin{enumerate}
\item $\phi$ is continuous and $\phi(\text{Fr}_F)$ is semisimple, 
 \item the restriction of $\phi$ to $\SL_2(\CC)$ is induced by a morphism of algebraic groups $\SL_2 \to \GL_n$
\end{enumerate}
If moreover the image of $W_F$ is bounded then we say that $\phi$ is tempered. Define $\phi^{\vee}$ by $\phi^\vee(w) = {}^t\phi(w)^{-1}$ and call this the contragredient representation of $\phi$. If $E/F$ is a quadratic extension of local fields and $\phi$ is a representation of $\WD_E$, fix $s \in W_F \smallsetminus W_E$ and define a representation $\phi^c$ of $\WD_E$ by $\phi^c(w) = \phi(sws^{-1})$.
The equivalence class of $\phi^c$ is independent of the choice of $s$. We say that $\phi$ is conjugate self-dual if there is an isomorphism $b: \phi \mapsto (\phi^{\vee})^{c}$. Note that there is a natural isomorphism $(((\phi^{\vee})^c)^{\vee})^c\simeq \phi$ so that $(b^{\vee})^c$ can be considered as an isomorphism from $\phi$ onto $(\phi^{\vee})^c$. For $\varepsilon=\pm 1$, we say that $\phi$ is conjugate self-dual with sign $\varepsilon$ if there exists such a $b$ satisfying the extra condition $(b^{\vee})^c=\varepsilon \cdot b$.
Define $\text{As}(\phi):WD_F \to \GL_{n^2}(\mathbb{C})$ by tensor induction of $\phi$ as follows;
$$\text{As}(\phi)(w)=\begin{cases}\phi(w)\otimes \phi(s^{-1}ws) \quad \quad  \quad \quad\text{ if } w\in WD_E \\ \iota \circ (\phi(s^{-1}w) \otimes \phi(ws)) \quad \quad \text{ if } w \in WD_F \smallsetminus WD_E\end{cases}$$ where $\iota$ is the linear isomorphism of $\mathbb{C}^n\otimes \mathbb{C}^n$ given by $\iota(x \otimes y )=y \otimes x$.

Then the equivalence class of $\text{As}(\phi)$ is also independent of the choice of $s$. We set $\text{As}^+(\phi)=\text{As}(\phi)$ and $\text{As}^{-}(\phi)=\text{As}(\chi  \otimes \phi)$. 

Let $V$ be a $n$-dimensional $\varepsilon$-hermitian space over $E$. An $L$-parameter for the unitary group $\U(V)$ is a $GL_n(\mathbb{C})$-conjugacy class of admissible homomorphisms $$\phi:\WD_F \longrightarrow ^LU(V)=GL_{n}(\CC) \rtimes \text{Gal}(E/F)$$ such that  the composite of $\phi$ with the projection onto Gal$(E/F)$ is the natural projection of $\WD_F$ to Gal$(E/F)$.

The following proposition from \cite[\S 8]{Gan2} enables us to removes the cumbersome Gal$(E/F)$-factor in the definition of $L$-parameter of $U(V)$.

\begin{prop}Restriction to $\WD_E$ gives a bijection between the set of $L$-parameters for $U(V)$ and the set of equivalence classes of conjugate self-dual representations $$\phi:\WD_E \longrightarrow GL_n(\CC)$$ 
of sign $(-1)^{n-1}$.
\end{prop}
With this proposition, by an $L$-parameter for $U(V)$, we shall mean a $n$-dimensional conjugate self-dual representation $\phi$ of $\WD_E$ of sign $(-1)^{n-1}$.

Given a $L$-parameter $\phi$ of $U(V)$, we say that  $\phi$ is generic if its adjoint $L$-function $L(s,\text{Ad} \circ \phi)=L(s,\text{As}^{(-1)^{n-1}} (\phi))$ is holomorphic at $s=1$.
 Write $\phi$ as a direct sum
\[  \phi = \bigoplus_i m_i \phi_i \]
of pairwise inequivalent irreducible representations $\phi_i$ of $\WD_E$ with multiplicities $m_i$.
We say that $\phi$ is square-integrable if it has no multiplicity (i.e. $m_i=1$ for all $i$) and $\phi_i$ is conjugate self-dual with sign $(-1)^{n-1}$ for all $i$. Furthermore, we can associate its component group $S_{\phi}$ to $\phi$.
As explained in \cite[\S 8]{Gan2}, $S_{\phi}$ is a finite 2-abelian group which can be described as 
\[  S_{\phi}  = \prod_j  (\ZZ / 2\ZZ) a_j \]
with a canonical basis $\{ a_j \}$, where the product ranges over all $j$ such that $\phi_j$ is conjugate self-dual with sign $(-1)^{n-1}$.
If we denote the image of $-1 \in \GL_n(\CC)$ in $S_\phi$ by $z_\phi$, we have\[
 z_{\phi} = (m_j a_j) \in \prod_j  (\ZZ / 2\ZZ) a_j.
\]

\subsection{Local Langlands correspondence for unitary group}

The aim of the local Langlands correspondence for unitary groups is to classify the irreducible smooth representations of unitary groups. To state it, we first introduce some notations.
\begin{itemize}
\item Let $V^+$ and $V^-$ be the $n$-dimensional $\varepsilon$-Hermitian spaces with $\epsilon(V^+) = +1$, $\epsilon(V^-) = -1$ respectively.
\item Let $\Irr(\U(V^{\pm}))$ be the set of irreducible smooth representations of $\U(V^{\pm})$.

\end{itemize}

Then the local Langlands correspondence a form enhenced by Vogan \cite{v}, says that for a $L$-parameter $\phi$ of $U(V^{\pm})$, there is the so-called Vogan $L$-packet  $\Pi_{\phi}$, a finite set consisting of irreducible smooth representations  of $U(V^{\pm})$,  such that 

\[  \Irr(\U(V^+)) \sqcup \Irr(\U(V^-))  = \bigsqcup_{\phi}  \Pi_{\phi}, \]
where $\phi$ on the right-hand side runs over all equivalence classes of $L$-parameters for $\U(V^\pm)$.
\noindent Then under the local Langlands correspondence, we may also decompose $\Pi_{\phi}$ as 
\[  \Pi_{\phi}  = \Pi_{\phi}^+\sqcup   \Pi_{\phi}^-, \]
where for $\epsilon = \pm 1$, $\Pi_{\phi}^{\epsilon}$ consists of the representations of $\U(V^{\epsilon})$ in $\Pi_{\phi}$.

Furthermore, as explained in \cite[\S 12]{Gan2}, there is a bijection $$J^{\psi}(\phi):\Pi_{\phi} \to \Irr( S_{\phi})$$ which is canonical when $n$ is odd and depends on the choice of an additive character of $\psi:F^{\times} \to \CC$ when $n$ is even. More precisely, such bijection is determined by the $\N_{E/F}(E^{\times})$-orbit of nontrivial additive characters 
\[
\begin{cases}
 \psi^E:E/F \rightarrow \CC^{\times} & \text{if $\varepsilon = +1$;} \\
 \psi:F \rightarrow \CC^{\times} & \text{if $\varepsilon = -1$.}
\end{cases}
\] 
According to this choice, when $n$ is even, we write
\[J^{\psi}=
\begin{cases}
 J_{\psi^E} & \text{if $\varepsilon = +1$;} \\
 J_{\psi} & \text{if $\varepsilon = -1$},
\end{cases} 
\] and even when $n$ is odd, we retain the same notation $J^{\psi}(\phi)$ for the canonical bijection.\\
Hereafter, if a nontrivial additive chracter $\psi:F \rightarrow \CC^{\times}$ is fixed, we define $\psi^E:E/F \rightarrow \CC^{\times} $ by
\[  \psi^E(x) : = \psi(\tfrac{1}{2} \Tr_{E/F}(\delta x)) \] and using these two characters, we fix once and for all a bijection $$J^{\psi}(\phi):\Pi_{\phi} \to \Irr( S_{\phi})$$ as above. 

With these fixed bijections, we can label all irreducible smooth representations of $\U(V^{\pm})$ as $\pi(\phi,\eta)$ for some unique $L$-parameter $\phi$ of $U(V^{\pm})$ and $\eta \in \Irr( S_{\phi})$.

\subsection{Properties of the local Langlands correspondence}
\label{SS:llc}

We briefly list some properties of the local Langlands correspondence for unitary group, which we will use in this paper:

\begin{itemize}
\item
$\pi(\phi,\eta)$ is a representation of $\U(V^{\epsilon})$ if and only if $\eta(z_\phi)  = \epsilon$.

\item $\pi(\phi,\eta)$ is tempered if and only if $\phi$ is tempered.

\item $\pi(\phi,\eta)$ is square-integrable if and only if $\phi$ is square-integrable.

\item The component groups $S_{\phi}$ and $S_{\phi^{\vee}}$ are canonically identified. Under this canonical identification, if $\pi=\pi(\phi,\eta)$, then its contragradient representation $\pi^{\vee}$ is $\pi(\phi^{\vee},\eta\cdot \nu)$ where \[  
\nu(a_j)  = \begin{cases}
\omega_{E/F}(-1)^{\dim \phi_j} & \text{if $\dim_{\CC} \phi$ is even;}  \\
1 & \text{if $\dim_{\CC} \phi$ is odd.} \end{cases} \]
(This property follows from a result of Kaletha \cite[Theorem 4.9]{kel}.)
\end{itemize}

\section{\textbf{Local theta correspondence}}
In this section, we state the local theta correspondence of unitary groups for two low rank cases. From now on, for $\epsilon = \pm 1$, we shall denote by $V_n^\epsilon$ the $n$-dimensional Hermitian space with $\epsilon(V_n^\epsilon) = \epsilon$ and by $W_n^\epsilon$ the $n$-dimensional skew-Hermitian space with $\epsilon(W_n^\epsilon) = \epsilon$, so that $W_n^\epsilon = \delta \cdot V_n^\epsilon$.

\subsection{The Weil representation for Unitary groups}
\label{SS}In this subsection, we introduce the Weil representation. 

Let $E/F$ be a quadratic extension of local fields and $(V_m,\langle,\rangle_{V_m})$ be a $m$-dimensional Hermitian space and $(W_n,\langle,\rangle_{W_n})$ a $n$-dimensional skew-Hermitian space over $E$. Define the symplectic space
$$
\mathbb{W}_{V_m,W_n} := \operatorname{Res}_{E/F} V_m \otimes_E W_n
$$
with the symplectic form $$
\langle v \otimes w,v' \otimes w' \rangle_{\mathbb{W}_{V_m,W_n}} := \frac{1}{2}\operatorname{tr}_{E/F}\left(\langle v,v'\rangle_{V_m}{\langle w,w' \rangle}_{W_n}\right).
$$
We also consider the associated symplectic group $Sp(\mathbb{W}_{V_m,W_n})$ preserving $\langle \cdot,\cdot \rangle_{\mathbb{W}_{V_m,W_n}}$ and the metaplectic group $\widetilde{Sp}(\mathbb{W}_{V_m,W_n})$ which sits in a short exact sequence :

$$1\to \CC^\times\to\widetilde{Sp}(\mathbb{W}_{V_m,W_n})\to Sp(\mathbb{W}_{V_m,W_m})\to 1.$$\\ 
Let $\mathbb{X}_{V_m,W_n}$ be a Lagrangian subspace of $\mathbb{W}_{V_m,W_n}$ and fix an additive character $\psi:F\to\CC^\times$. Then we have a Schr\"odinger model of the Weil Representation $\omega_\psi$ of $\widetilde{Sp}(\mathbb{W})$ on $\mathcal{S}(\mathbb{X}_{V_m,W_n})$, where $\mathcal{S}$ is the Schwartz-Bruhat function space.

If we set  \[ \chi_{V_m} :=\chi^m \quad \text{ and } \quad \chi_{W_n}:=\chi^n, \]
where $\chi$ is a character of $E^{\times}$ whose resriction to $F^{\times}$ is $\omega_{E/F}$, which is the quadratic character assosiated to $E/F$ by local class field theory, then $(\chi_{V_m}, \chi_{W_n})$ gives a splitting homomorphism $$
\iota_{\chi_{V_m}, \chi_{W_n}}:U(V_m)\times U(W_n)\to \widetilde{Sp}(\mathbb{W}_{V_m,W_n})
$$ and so by composing this to $\omega_\psi$, we have a Weil representation $\omega_\psi\circ \iota_{\chi_{V_m},\chi_{W_n}}$ of $U(V_m) \times U(W_n)$ on $\mathbb{S}(\mathbb{X}_{V_m,W_n})$.

When the choice of $\psi$ and $(\chi_{V_m}, \chi_{W_n})$ is fixed as above, we simply write $$\omega_{\psi,W_n,V_m}:= \omega_\psi\circ \iota_{\chi_{V_m},\chi_{W_n}}.$$ Throughout the rest of the paper, when it comes to a Weil representations of $U(V_m) \times U(W_n)$, we shall denote it by $\omega_{\psi,W_n,V_m}$ with understanding the choices of characters $(\chi_{V_m},\chi_{W_n})$  as above.

\begin{rem}\label{con} When $m=1$, the image of $U(V_1)$ in $\widetilde{Sp}(\mathbb{W}_{V_1,W_n})$ coincides with the image of the center of $U(W_n)$, and so we regard the Weil representation of $U(V_1) \times U(W_n)$  as a representation of $U(W_n)$. In this case, we denote the Weil representation of $U(W_n)$ as $\omega_{\psi,W_n}$. Furthermore, we can also use $\chi_{V_1}=\chi^{-1}$ for the choice of splitting homomorphism $\iota_{\chi_{V_1}, \chi_{W_n}}$ instead of $\chi_{V_1}=\chi$. In this case, the Weil representation of $U(W_n)$ is $\omega^{\vee}_{\psi,W_n}$.
\end{rem}

\subsection{Local theta correspondence}

Given a Weil representation $\omega_{\psi,W_n,V_m}$ of $U(V_m) \times U(W_n)$ and an irreducible smooth representation $\pi$ of $\U(W_n)$, the maximal $\pi$-isotypic quotient of $\omega_{\psi,V_m,W_n}$ is of the form
\[
 \Theta_{\psi,V_m,W_n}(\pi) \boxtimes \pi
\]
for some smooth representation $\Theta_{\psi,V_m,W_n}(\pi)$ of $\U(V_m)$ of finite length.
By the Howe duality\footnote{It was first proved by Waldspurger \cite{w} for all residual characteristic except $p= 2$. Recently,  Gan and Takeda \cite{gt1}, \cite{gt2} have made it available for all residual characteristics.}, the maximal semisimple quotient $\theta_{\psi,V_m,W_n}(\pi)$ of $\Theta_{\psi,V_m,W_n}(\pi)$ is either zero or irreducible. 

In this paper, we consider two kinds of theta correspondences for $(\U(1) \times \U(3))$ and $(\U(2) \times \U(3))$ :

\subsection{Case (i)} Now we shall consider the theta correspondence for $\U(V_{1}^\epsilon) \times \U(W_3^{\epsilon'})$. The following is a compound of Theorem 3.4 and Theorem 3.5 in \cite{Haan}.
\begin{thm}\label{Te}Let $\phi$ be a $L$-parameter of $U(V_1^{\pm})$. Then we have:

\begin{enumerate}
\item For any $\epsilon,\epsilon'=\pm1$ and any $\pi \in \Pi_{\phi}^{\epsilon'}$, $\Theta_{\psi, W_3^{\epsilon},V_1^{\epsilon'}}(\pi)$ is nonzero and irreducible.

\item $\Theta_{\psi, W_3^{\epsilon},V_1^{\epsilon'}}(\pi)$ $=\begin{cases} \text{a non-tempered representation} & \text{if $\epsilon(\frac{1}{2},\phi \otimes \chi^{-3}, \psi_2^{E})=\epsilon \cdot \epsilon'$} \\ \text{a supercuspidal representation } & \text{if  $\epsilon(\frac{1}{2},\phi \otimes \chi^{-3}, \psi_2^{E})=-\epsilon \cdot \epsilon'$} \end{cases} $ , \\ where $$\psi_2^{E}(x)=\psi (\Tr_{E/F}(\delta x)).$$

\item The $L$-parameter $\theta(\phi)$ of $\Theta_{\psi, W_3^{\epsilon},V_1^{\epsilon'}}(\pi)$ has the following form; \begin{equation}\label{non}\theta(\phi)=\begin{cases} \theta^n(\phi)=\chi  |\cdot|_{E}^{\frac{1}{2}} \oplus \phi \cdot \chi^{-2} \oplus \chi  |\cdot|_{E}^{-\frac{1}{2}}& \text{if $\epsilon(\frac{1}{2},\phi \otimes \chi^{-3}, \psi_2^{E})=\epsilon \cdot \epsilon'$} \\ \theta^s(\phi)= \phi \cdot \chi^{-2} \oplus \chi \boxtimes \textbf{S}_2 &  \text{if $\epsilon(\frac{1}{2},\phi \otimes \chi^{-3}, \psi_2^{E})=-\epsilon \cdot \epsilon'$}  \end{cases},\end{equation} where $\textbf{S}_2$ is the standard 2-dimensional representation of $SL_2(\mathbb{C}).$

\item For $\epsilon, \epsilon'$ such that $\epsilon(\frac{1}{2},\phi \otimes \chi^{-3}, \psi_2^{E})=\epsilon \cdot \epsilon'$, the theta correspondence $\pi \mapsto \theta_{\psi, W_{3}^{\epsilon}, V_1^{\epsilon'}}(\pi)$ gives a bijection
\[  \Pi_{\phi} \longleftrightarrow \Pi_{\theta^n(\phi)}.  \]  

\item For $\epsilon, \epsilon'$ such that $\epsilon(\frac{1}{2},\phi \otimes \chi^{-3}, \psi_2^{E})=-\epsilon \cdot \epsilon'$, the theta correspondence $\pi \mapsto \theta_{\psi, W_{3}^{\epsilon}, V_1^{\epsilon'}}(\pi)$ gives an injection
\[  \Pi_{\phi} \hookrightarrow \Pi_{\theta^s(\phi)}.  \]  

\end{enumerate}
Write \begin{itemize} \item $S_{\phi}\ = \ (\ZZ/2\ZZ)a_1$
\item $S_{\theta^n(\phi)}=(\ZZ/2\ZZ)a_1 \quad \ \quad\quad\quad\quad\quad \text{ if  }\ \epsilon(\frac{1}{2},\phi \otimes \chi^{-3}, \psi_2^{E})=\epsilon \cdot \epsilon'$
 \item $S_{\theta^s(\phi)}=(\ZZ/2\ZZ)a_1   \times (\ZZ/2\ZZ)a_2 
 \quad \text{ if  }\ \epsilon(\frac{1}{2},\phi \otimes \chi^{-3}, \psi_2^{E})=-\epsilon \cdot \epsilon' $\end{itemize}
where $$\psi_2^{E}(x)=\psi (\Tr_{E/F}(\delta x)).$$
(Note that $\theta^s(\phi)$ is a square-integrable $L$-parameter of $U(W_3^{\epsilon})$ and the summand $(\ZZ/2\ZZ)a_2$ of $S_{\theta^s(\phi)}$ arises from the summand $\chi \boxtimes \textbf{S}_2$ in $\theta^s(\phi)$.)

Since we are only dealing with odd dimensional spaces, there are three canonical bijections

\begin{itemize}
\item $J^{\psi}(\phi): \Pi_{\phi} \longleftrightarrow \Irr(S_{\phi}) $

\item $J^{\psi}(\theta^n(\phi)): \Pi_{\theta^n(\phi)} \longleftrightarrow \Irr(S_{\theta^n(\phi)}) $

\item $J^{\psi}(\theta^s(\phi)): \Pi_{\theta^s(\phi)} \longleftrightarrow \Irr(S_{\theta^s(\phi)}).$

\end{itemize}

Using these maps, the following bijection and inclusion
\begin{align*}
 \Irr(S_{\phi}) & \longleftrightarrow \Irr(S_{\theta^n(\phi)}) \\
 \eta & \mapsto \theta^n(\eta),
\end{align*}  
\begin{align*}
 \Irr(S_{\phi}) & \hookrightarrow \Irr(S_{\theta^s(\phi)}) \\
 \eta & \mapsto \theta^s(\eta)
\end{align*}induced by the theta correspondence can be explicated as follows:
\begin{equation}\label{theta1} \theta^{n}(\eta)(a_1)=\eta(a_1) \cdot \epsilon(\frac{1}{2},\phi \otimes \chi^{-3},\psi_{2}^E ),\end{equation}
 \begin{equation}\label{theta2}\theta^{s}(\eta)(a_1)=\eta(a_1) \cdot \epsilon(\frac{1}{2},\phi \otimes \chi^{-3},\psi_{2}^E ) , \\ \quad \quad \theta_{2}(\eta)(a_2)=-1.\end{equation}
\end{thm}

\begin{rem}Note that $\theta^n(\phi)$ is a non-generic $L$-parameter. Gan and Ichino (Proposition B.1 in Appendix in \cite{iw}) proved that an $L$-parameter is generic if and only if its associated $L$-packet $\Pi_{\phi}$ contains a generic representation (i.e. one possessing a Whittaker model). Together with the Corollary 4.2.3 in \cite{Ge2}, which asserts that all elements in $\Pi_{\theta^n(\phi)}$ have no Whittaker models, we see that $\theta^n(\phi)$ is a non-generic $L$-parameter.

\end{rem}
\subsection{Case (ii)}

Now we shall consider the theta correspondence for $\U(V_2^{\epsilon'}) \times \U(W_{3}^\epsilon)$. The following summarises some results of \cite{gi}, \cite{iw}, which are specialised to this case. 

\begin{thm}  \label{Tae}

Let $\phi$ be an $L$-parameter for $\U(V_2^{\pm})$. Then we have:

\begin{enumerate}

\item Suppose that $\phi$ does not contain $\chi^3$.

\begin{enumerate}

 \item For any $\pi \in \Pi_{\phi}^{\epsilon'}$ and any $\epsilon \in \{ \pm1\}$, 
 $\Theta_{\psi, W_{3}^{\epsilon}, V_2^{\epsilon'}}(\pi)$ is nonzero and $\theta_{\psi, W_{3}^{\epsilon}, V_2^{\epsilon'}}(\pi)$ has $L$-parameter
 \[  \theta(\phi) = (\phi \otimes \chi^{-1}) \oplus  \chi^2. \]

 \item For each $\epsilon = \pm 1$, the theta correspondence $\pi \mapsto \theta_{\psi, W_{3}^{\epsilon}, V_2^{\epsilon'}}(\pi)$ gives a bijection
\[  \Pi_{\phi} \longleftrightarrow \Pi_{\theta(\phi)}^{\epsilon}.  \]  

\end{enumerate} 

\item Suppose that $\phi$ contains $\chi^3$.

\begin{enumerate}

\item For any $\pi \in \Pi_{\phi}^{\epsilon'}$, exactly one of $\Theta_{\psi, W_{3}^+, V_2^{\epsilon'}}(\pi)$ or $\Theta_{\psi, W_{3}^-, V_2^{\epsilon'}}(\pi)$ is nonzero.

\item If $\Theta_{\psi, W_{3}^{\epsilon}, V_2^{\epsilon'}}(\pi)$ is nonzero, then $\theta_{\psi, W_{3}^{\epsilon}, V_2^{\epsilon'}}(\pi)$ has $L$-parameter
 \[  \theta(\phi) = (\phi \otimes \chi^{-1} ) \oplus  \chi^2. \]

\item The theta correspondence $\pi \mapsto \theta_{\psi, W_{3}^{\epsilon}, V_2^{\epsilon'}}(\pi)$ gives a bijection
 \[  \Pi_{\phi} \longleftrightarrow \Pi_{\theta(\phi)}. \]

\end{enumerate} 

\item We have fixed a bijection 
 \[ J_{\psi^E}(\phi): :  \Pi_{\phi} \longleftrightarrow \Irr(S_{\phi}) ,\] where 
\[  \psi^E(x)  = \psi(\tfrac{1}{2} \Tr_{E/F}(\delta x)) \]
and there is the bijection

\[ \quad   J^{\psi}(\theta(\phi)):  \Pi_{\theta(\phi)} \longleftrightarrow \Irr(S_{\theta(\phi)}). \]
\begin{itemize}

\item If $\phi$ does not contain $\chi^3$, we have
\[  \ S_{\theta(\phi)}  = S_{\phi} \times (\ZZ/2\ZZ)b_1, \]
where the extra copy of $\ZZ/2\ZZ$ of $S_{\theta(\phi)}$ arises from the summand $\chi^2$ in $\theta(\phi)$. \\Then for each $\epsilon$, using the above bijection $J$ and $J_{\psi^E}$, one has a canonical bijection
\begin{align*}
 \Irr(S_{\phi}) & \longleftrightarrow \Irr^{\epsilon}(S_{\theta(\phi)}) \\
 \eta & \longleftrightarrow \theta(\eta)
\end{align*}
induced by the theta correspondence, where $\Irr^{\epsilon}(S_{\theta(\phi)})$ is the set of irreducible characters $\eta'$ of $S_{\theta(\phi)}$ such that $\eta'(z_{\theta(\phi)}) =\epsilon$ and the bijection is determined by
 \[  \theta(\eta)|_{S_{\phi}}  =\eta. \]

\item If $\phi$ contains $\chi^3$, then $\phi \otimes \chi^{-1}$ contains $\chi^2$, and so
 \[
   S_{\theta(\phi)}  =  S_{\phi}. \]
Thus, one has a canonical bijection
\begin{align*}
  \Irr(S_{\phi}) & \longleftrightarrow \Irr(S_{\theta(\phi)}) \\
  \eta & \longleftrightarrow \theta(\eta)
\end{align*}
induced by the theta correspondence and it is given by \[ \theta(\eta) = \eta. \]

\end{itemize}

\item If $\pi$ is tempered and $\Theta_{\psi, W_{3}^{\epsilon}, V_2^{\epsilon'}}(\pi)$ is nonzero, then $\Theta_{\psi, W_{3}^{\epsilon}, V_2^{\epsilon'}}(\pi)$ is irreducible and tempered.

\end{enumerate}

\end{thm}

\section{\textbf{Main Theorem}}
In this section, we prove our main theorem. To prove it, we first state the precise result of Beuzart-Plessis which we shall use in the proof of Theorem \ref{thm2}.\begin{footnote}{Recently, Gan and Ichino (\cite{iw}) extended Beuzart-Plessis's work to the generic case relating it to (FJ) case.}\end{footnote} 

\begin{B}Let $\phi=\phi^{(n+1)} \times \phi^{(n)}$ be a tempered $L$-parameter of $U(V_{n+1}^{\pm}) \times U(V_n^{\pm})$ and write $S_{\phi^{(n+1)}}=\prod_{i} (\ZZ/2\ZZ)a_i$ and $S_{\phi^{(n)}}=\prod_{j} (\ZZ/2\ZZ)b_j$. Let $\Delta : U(V_n^{\pm})  \hookrightarrow U(V_{n+1}^{\pm}) \times U(V_n^{\pm})$ be the diagonal map. Then for $\pi(\eta) \in \Pi_{\phi}^{R,\ \pm}=\Pi_{\phi^{(n+1)}}^{\pm} \times \Pi_{\phi^{(n)}}^{\pm} $ where $\eta \in \Irr(S_{\phi})= \Irr (S_{\phi^{(n+1)}}) \times \Irr (S_{\phi^{(n)}})$, 
$$\Hom_{\Delta (U(V_n^{\pm}))} ( \pi(\eta), \CC )=1 \Longleftrightarrow \eta=\eta^{\spadesuit} \text{ where }$$
\[
\begin{cases}
 \eta^{\spadesuit}(a_i)  = \epsilon(\frac{1}{2}, \phi^{(n+1)}_i \otimes \phi^{(n)}, \psi^E_{-2}); \\
 \eta^{\spadesuit}(b_j)  = \epsilon(\frac{1}{2},  \phi^{(n+1)} \otimes \phi^{(n)}_{j}, \psi^E_{-2}).
\end{cases}
\] where $\psi^E_{-2}(x)  = \psi(-\Tr_{E/F}(\delta x))$.
\end{B}

\begin{thm}\label{thm2} Let $\phi^{(1)},\phi^{(2)}$ be tempered $L$-parameters of $U(V_1^{\pm})$ and $U(V_2^{\pm})$ respectively and suppose that $\phi^{(2)}$ does not contain $\chi^{-3}$. Let $$\theta^n(\phi^{(1)})=\chi  |\cdot|_{E}^{\frac{1}{2}} \oplus \phi^{(1)} \cdot \chi^{-2} \oplus \chi  |\cdot|_{E}^{-\frac{1}{2}},$$ $$\theta^s(\phi^{(1)})=\phi^{(1)} \cdot \chi^{-2} \oplus \chi \boxtimes \textbf{S}_2$$ be the two $L$-parameters of $U(W_3^{\pm})$ appearing in (\ref{non}) and let $$\theta(\phi^{(2)})=\phi^{(2)} \otimes \chi \oplus \chi^{-2} $$ be the $L$-parameters of $U(W_3^{\pm})$ appearing in Theorem \ref{Tae} (ii), in which $\chi$ is replaced by $\chi^{-1}$.

Write  \begin{itemize} 
\item $S_{\phi^{(1)}}=S_{\theta^n(\phi^{(1)})}=(\ZZ/2\ZZ)a_1;$ \item $S_{\theta^s(\phi^{(1)})}=S_{\phi^{(1)}} \times (\ZZ/2\ZZ)a_2;$

\item $S_{\phi^{(2)}}=\begin{cases}(\ZZ/2\ZZ)b_1 \quad \text{ if $\phi^{(2)}$ is irreducible}; \\ (\ZZ/2\ZZ)b_1 \times (\ZZ/2\ZZ)b_2 \quad \text{ if $\phi^{(2)}=\phi_1^{(2)} \oplus \phi_2^{(2)}$ is reducible} \end{cases}$

\item $S_{\theta(\phi^{(2)})}=S_{\phi^{(2)}} \times (\ZZ/2\ZZ)c_1$ \end{itemize} where $c_1$ comes from the component $\chi^{-2}$ of $\theta(\phi^{(2)})$.\\
We use the fixed character $\psi$  to fix the local Langlands correspondence for $\Pi_{ \phi^{(2)}} \leftrightarrow \Irr(S_{\phi^{(2)}}).$

\noindent For $x=n, s$, let $$\theta^{x}(\phi^{(1)},\phi^{(2)})=\theta^x(\phi^{(1)}) \times \theta(\phi^{(2)})$$ be a $L$-parameter of $G_3^{\pm}=\U(W_3^\pm) \times \U(W_3^\pm)$ and $\pi^{x}(\eta)\in \Pi^{R, \ \epsilon}_{\theta^{x}(\phi^{(1)},\phi^{(2)})}$ a representation of a relevant pure inner form of $G_3$.
Then, 
\[ \Hom_{\Delta \U(W_3^\epsilon)}(\pi^{x}(\eta), \omega_{\psi,W_3^{\epsilon}}) \ne 0 \Longleftrightarrow \eta = \eta_{x}^{\blacklozenge}\]
where $\eta_{x}^{\blacklozenge}\in\Irr(S_{\theta^{x}(\phi^{(1)},\phi^{(2)})}) = \Irr(S_{\theta^x(\phi^{(1)})}) \times  \Irr(S_{\theta(\phi^{(2)})}) $ is specified as follows;
\begin{enumerate}
\item When $\phi^{(2)}$ is irreducible, 
$$\begin{cases}\eta_{n}^{\blacklozenge}(a_1)=\epsilon(\frac{1}{2},\phi^{(1)} \otimes \chi^{-3},\psi_2^{E}) \cdot \epsilon(\frac{1}{2},\phi^{(1)} \otimes \phi^{(2)}, \psi_2^{E})\\ \eta_{n}^{\blacklozenge}(b_1)=\epsilon(\frac{1}{2},\phi^{(1)} \otimes \phi^{(2)},\psi_2^{E})\\ \eta_{n}^{\blacklozenge}(c_1)=\epsilon(\frac{1}{2},\phi^{(1)} \otimes \chi^{-3},\psi_2^{E}) \end{cases} ,$$ 
$$\begin{cases}\eta_{s}^{\blacklozenge}(a_1)=\epsilon(\frac{1}{2},\phi^{(1)} \otimes \chi^{-3},\psi_2^{E}) \cdot \epsilon(\frac{1}{2},\phi^{(1)} \otimes \phi^{(2)}, \psi_2^{E})\\ \eta_{s}^{\blacklozenge}(a_2)=-1\\

\eta_{s}^{\blacklozenge}(b_1)=\epsilon(\frac{1}{2},\phi^{(1)} \otimes \phi^{(2)},\psi_2^{E})\\ \eta_{s}^{\blacklozenge}(c_1)=-\epsilon(\frac{1}{2},\phi^{(1)} \otimes \chi^{-3},\psi_2^{E}) \end{cases} .$$\\

\item When $\phi^{(2)}=\phi_1^{(2)}\oplus \phi_2^{(2)}$ is reducible,
$$\begin{cases}\eta_{n}^{\blacklozenge}(a_1)=\epsilon(\frac{1}{2},\phi^{(1)} \otimes \chi^{-3},\psi_2^{E}) \cdot \epsilon(\frac{1}{2},\phi^{(1)} \otimes \phi^{(2)}, \psi_2^{E})\\ \eta_{n}^{\blacklozenge}(b_1)=\epsilon(\frac{1}{2}, \phi^{(1)} \otimes \phi_1^{(2)}, \psi_2^{E})\\ \eta_{n}^{\blacklozenge}(b_2)=\epsilon(\frac{1}{2}, \phi^{(1)} \otimes \phi_2^{(2)}, \psi_2^{E}) \\ \eta_{n}^{\blacklozenge}(c_1)=\epsilon(\frac{1}{2},\phi^{(1)} \otimes \chi^{-3},\psi_2^{E}) \end{cases},$$ 
$$\begin{cases}\eta_{s}^{\blacklozenge}(a_1)=\epsilon(\frac{1}{2},\phi^{(1)} \otimes \chi^{-3},\psi_2^{E}) \cdot \epsilon(\frac{1}{2},\phi^{(1)} \otimes \phi^{(2)}, \psi_2^{E})\\ \eta_{s}^{\blacklozenge}(a_2)=-1\\

\eta_{s}^{\blacklozenge}(b_1)=\epsilon(\frac{1}{2},\phi^{(1)} \otimes \phi_1^{(2)}, \psi_2^{E})\\ \eta_{s}^{\blacklozenge}(b_2)=\epsilon(\frac{1}{2}, \phi^{(1)} \otimes \phi_2^{(2)}, \psi_2^{E}) \\ \eta_{s}^{\blacklozenge}(c_1)=-\epsilon(\frac{1}{2},\phi^{(1)} \otimes \chi^{-3},\psi_2^{E}) \end{cases}.$$
\end{enumerate}
Furthermore, \[ \dim_{\mathbb{C}}\Hom_{\Delta \U(W_3^\epsilon)}(\pi^{x}(\eta_{x}^{\blacklozenge}), \omega_{\psi,W_3^{\epsilon}}) =1 \]
\end{thm}

\begin{rem}When $x=n$ or $s$, we have $\eta_{x}^{\blacklozenge}(z_{\theta^x(\phi^{(1)})})=\eta_{x}^{\blacklozenge}(z_{\theta(\phi^{(2)})})$ so that $\eta_{x}^{\blacklozenge}$ always corresponds to a representation $\pi^x(\eta_{x}^{\blacklozenge})$ of a relevant pure inner form of $G_3$.

\end{rem}

\begin{proof} For each $x=n,s$, we first prove the existence of some $\epsilon_x \in \{\pm1\}$ and $\pi^{x}(\eta)\in \Pi^{R,\ \epsilon_x}_{\theta^{x}(\phi^{(1)},\phi^{(2)})}$ such that $$\Hom_{\Delta \U(W_3^{\epsilon_x})}(\pi^{x}(\eta), \omega_{\psi,W_3^{\epsilon_x}}) \ne 0.$$
For $a\in F^{\times}$, let $L_a$ be the 1-dimensional Hermitian space with form $a \cdot \N_{E/F}$. Then
\[ V_{2}^{+}/V_1^{+} \simeq V_{2}^{-}/V_1^{-} \simeq L_{-1}.
\]
We consider the following see-saw diagram : ($\epsilon, \epsilon'$ will be determined soon)
\begin{equation}\label{ss}
 \xymatrix{
  \U(W_3^{\epsilon})  \times \U(W_3^{\epsilon})  \ar@{-}[dr] \ar@{-}[d] & \U(V^{\epsilon'}_{2}) \ar@{-}[d] \\
  \U(W_3^{\epsilon}) \ar@{-}[ur] &  \U(V^{\epsilon'}_1) \times \U(L_{-1})}.
\end{equation}
In this diagram, we shall use three theta correspondences : 
\begin{enumerate}
\item $U(V_2^{\epsilon'}) \times U(W_3^{\epsilon})$ relative to the pair of characters $(\chi^2, \chi^3)$;
\item $\U(V_1^{\epsilon'}) \times U(W_3^{\epsilon})$ relative to the pair of characters $(\chi, \chi^3)$;
\item $\U(L_{-1}) \times \U(W_3^{\epsilon})$ relative to the pair of characters $(\chi^{-1}, \chi^3)$.

\end{enumerate}

\noindent By $(B)_1$, there is a unique $\epsilon'\in \{\pm1\}$ and a unique pair of components characters $$ (\eta_2 ,\eta_1) \in \Irr^{\epsilon'}({S}_{(\phi^{(2)})^{\vee}}) \times \Irr^{\epsilon'}({S}_{(\phi^{(1)})^{\vee}})$$ such that $$\Hom_{\Delta U(V_1^{\epsilon'})}(\pi(\eta_2)\otimes \pi(\eta_1),\mathbb{C})\ne 0.$$
Moreover, $\epsilon'=\eta_1(a_1)=\epsilon(\frac{1}{2},(\phi^{(1)})^{\vee} \otimes (\phi^{(2)})^{\vee}, \psi^E_{-2})=\epsilon(\frac{1}{2},\phi^{(1)} \otimes \phi^{(2)}, \psi_2^{E})$.

By Theorem \ref{Tae} (i), (iv) and Theorem 4.1 in \cite{ag},$\Theta_{\psi,V_2^{\epsilon'},W_3^{\epsilon}}\big(\Theta_{\psi,W_3^{\epsilon},V_2^{\epsilon'}}(\pi(\eta_2))\big)$ is nonzero  for any $\epsilon \in \{\pm1\}$. Since $\Theta_{\psi,W_3^{\epsilon},V_2^{\epsilon'}}(\pi(\eta_2)) \boxtimes \Theta_{\psi,V_2^{\epsilon'},W_3^{\epsilon}}\big(\Theta_{\psi,W_3^{\epsilon},V_2^{\epsilon'}}(\pi(\eta_2))\big)$ is the maximal $\Theta_{\psi,W_3^{\epsilon},V_2^{\epsilon'}}(\pi(\eta_2))$-isotypic quotient of $\omega_{\psi,V_2,W_3}$ and $\Theta_{\psi,W_3^{\epsilon},V_2^{\epsilon'}}(\pi(\eta_2)) \boxtimes \pi(\eta_2)$ is a $\Theta_{\psi,W_3^{\epsilon},V_2^{\epsilon'}}(\pi(\eta_2))$-isotypic quotient of $\omega_{\psi,V_2,W_3}$, $\pi(\eta_2)$ should be a quotient of $\Theta_{\psi,V_2^{\epsilon'},W_3^{\epsilon}}\big(\Theta_{\psi,W_3^{\epsilon},V_2^{\epsilon'}}(\pi(\eta_2))\big)$. By Proposition 5.4 in \cite{ag}, $\Theta_{\psi,V_2^{\epsilon'},W_3^{\epsilon}}\big(\Theta_{\psi,W_3^{\epsilon},V_2^{\epsilon'}}(\pi(\eta_2))\big)$ is irreducible and thus we have $\Theta_{\psi,V_2^{\epsilon'},W_3^{\epsilon}}\big(\Theta_{\psi,W_3^{\epsilon},V_2^{\epsilon'}}(\pi(\eta_2))\big)=\pi(\eta_2)$.

Since $$\Hom_{\Delta U(V_1^{\epsilon'})}(\pi(\eta_2), \pi^{\vee}(\eta_1))\ne 0,$$  by the see-saw identity and Remark \ref{con}, we have $$\Hom_{\Delta U(W_3^{\epsilon})}\big(\Theta_{\psi,W_3^{\epsilon},V_1^{\epsilon'}}(\pi^{\vee}(\eta_1)) \otimes \omega^{\vee}_{\psi,W_3^{\epsilon}}, \Theta_{\psi,W_3^{\epsilon},V_2^{\epsilon'}}(\pi(\eta_2))\big)\ne 0$$ and since
$\Theta_{\psi,W_3^{\epsilon},V_2^{\epsilon'}}(\pi(\eta_2))$ and $\omega^{\vee}_{\psi,W_3^{\epsilon}}$ are admissible, we have $$\Hom_{\Delta U(W_3^{\epsilon})}\big(\Theta_{\psi,W_3^{\epsilon},V_1^{\epsilon'}}(\pi^{\vee}(\eta_1)) \otimes \Theta^{\vee}_{\psi,W_3^{\epsilon},V_2^{\epsilon'}}(\pi(\eta_2)) , \omega_{\psi,W_3^{\epsilon}} \big)\ne 0.$$\\
By Theorem \ref{Te} (iii) and Theorem \ref{Tae} (i), the $L$-parameter of $ \Theta_{\psi,W_3^{\epsilon},V_1^{\epsilon'}}(\pi^{\vee}(\eta_1)) \otimes \Theta^{\vee}_{\psi,W_3^{\epsilon},V_2^{\epsilon'}}(\pi(\eta_2)) $ is \[ \begin{cases} \theta^{n}(\phi^{(1)},\phi^{(2)}) \quad  \text{ if } \epsilon = \epsilon(\frac{1}{2},\phi^{(1)} \otimes \chi^{-3}, \psi_2^{E})\cdot \epsilon(\frac{1}{2},\phi^{(1)} \otimes \phi^{(2)}, \psi_2^{E})\\ \theta^{s}(\phi^{(1)},\phi^{(2)}) \quad  \text{ if } \epsilon = -\epsilon(\frac{1}{2},\phi^{(1)} \otimes \chi^{-3}, \psi_2^{E})\cdot \epsilon(\frac{1}{2},\phi^{(1)} \otimes \phi^{(2)}, \psi_2^{E}) \end{cases} \]
and by Theorem \ref{Te} (v), Theorem \ref{Tae} (iii), we see that their associated component characters are $\eta_{n}^{\blacklozenge}, \eta_{s}^{\blacklozenge}$ in each cases.

Next we shall prove that these are the unique representations which yield \emph{Fourier-Jacobi} model in each $L$-packets $\Pi_{\theta^{n}(\phi^{(1)},\phi^{(2)})}$ and $\Pi_{\theta^{s}(\phi^{(1)},\phi^{(2)})}$.\

\noindent Since $\theta^{s}(\phi^{(1)},\phi^{(2)})$ is tempered $L$-parameter, the uniqueness easily follows from (FJ)$_3$ in this case. Therefore, we shall only consider the non-tempered $L$-parameter $\theta^{n}(\phi^{(1)},\phi^{(2)})$.\\
Let $\pi_2 \otimes \pi_1 \in \Pi^{R, \ \epsilon}_{\theta^{n}(\phi^{(1)},\phi^{(2)})}=\Pi^{\epsilon}_{\theta^{n}(\phi^{(1)})}\times \Pi^{\epsilon}_{\theta(\phi^{(2)})}$ be a representation satisfying $$\Hom_{\Delta \U(W_3^{\epsilon})}(\pi_2 \otimes \pi_1, \omega_{\psi,W_3^{\epsilon}}) \ne 0$$ and in turn $$\Hom_{\Delta \U(W_3^{\epsilon})}(\pi_2 \otimes \omega^{\vee}_{\psi,W_3^{\epsilon}} , \pi^{\vee}_1)\ne0.$$

\noindent (The existence of such $\pi_2 \otimes \pi_1$ was ensured by the previous step.)\\
By Theorem \ref{Te} (iv), we can write $\pi_2=\Theta_{\psi,W_3^{\epsilon},V_1^{\epsilon'}}(\pi^{(1)})$ for some $\pi^{(1)}\in \Pi_{\theta^{(1)}}^{\epsilon'}$ where $$\epsilon'=\epsilon \cdot \epsilon(\frac{1}{2},\phi^{(1)} \otimes \chi^{-3}, \psi_2^{E}).$$
Then by applying the see-saw duality in the see-saw diagram in (\ref{ss}), one has
$$\Hom_{\Delta \U(W_3^{\epsilon})}(\pi_2 \otimes \omega^{\vee}_{\psi,W_3^{\epsilon}} , \pi^{\vee}_1) \simeq \Hom_{U(V_1^{\epsilon'})}(\pi^{(2)},\pi^{(1)})\ne0$$ where $\pi^{(2)}=\Theta_{\psi,V_2^{\epsilon'},W_3^{\epsilon}}(\pi^{\vee}_1)$. \\Note that $\pi^{(2)} \ne0$ and so it has tempered $L$-parameter $(\phi^{(2)})^{\vee}$. \\Then by the (B)$_1$, $(\pi^{(2)},\pi^{(1)})$ is the unique pair in the $L$-packet $\Pi_{(\phi^{(2)})^{\vee}} \times \Pi_{\phi^{(1)}}$ such that $$\Hom_{U(V_1^{\epsilon'})}(\pi^{(2)},\pi^{(1)})\ne0$$ and so $(\pi_2,\pi_1)$ should be $(\Theta_{\psi,W_3^{\epsilon},V_1^{\epsilon'}}(\pi^{(1)}),\Theta^{\vee}_{\psi,W_3^{\epsilon},V_2^{\epsilon'}}(\pi^{(2)}))$.\\ This settles down the uniqueness issue.

\end{proof}

\begin{rem}When the $L$-parameter $\phi^{(2)}$ of $U(V_2^{\pm})$ contains $\chi^{-3}$, we can write $\phi^{(2)}=\phi_0 \oplus \chi^{-3}$ for a $L$-parameter $\phi_0$ of $U(V_1^{\pm})$. Then, we set $$\theta(\phi^{(2)})=\begin{cases}3\cdot \chi^{-2} \quad \quad \quad \quad \ \ \text{ if $\phi_0=\chi^{-3}$ }, \\ \phi_0 \cdot \chi\oplus 2 \cdot \chi^{-2} \quad \text{ if } \phi_0\ne \chi^{-3}\end{cases}$$ and $$S_{\theta(\phi^{(2)})}=\begin{cases}(\ZZ/2\ZZ)b_1 \ \quad \quad \quad  \quad \quad \quad  \text{ if $\phi_0=\chi^{-3}$ },   \\ (\ZZ/2\ZZ)b_1 \times (\ZZ/2\ZZ)c_1  \quad \text{ if } \phi_0\ne \chi^{-3}.\end{cases}$$

If one develops a similar argument in this case, one could also have a recipe as in Theorem \ref{mul} for non-tempered case. However, we need some assumption on the irreducibility of the theta lifts because we cannot apply Proposition 5.4 in \cite{ag}.

\end{rem}

\begin{thm}\label{mul} Let the notations be as in Thorem \ref{thm2} and assume this time that $\phi^{(2)}$ contains $\chi^{-3}$.\\
Assume that for $\pi \in \prod^{\epsilon}_{\theta((\phi^{(2)})^{\vee})}$, if $\Theta_{\psi,V_2^{\epsilon'},W_3^{\epsilon}}(\pi)$ is nonzero, it is irreducible.

Then,
\[ \Hom_{\Delta \U(W_3^\epsilon)}(\pi^{n}(\eta), \omega_{\psi,W_3^{\epsilon}}) \ne 0 \Longleftrightarrow \eta = \eta_{n}^{\blacklozenge}\]
where $\eta_{n}^{\blacklozenge}\in\Irr(S_{\theta^{n}(\phi^{(1)},\phi^{(2)})}) = \Irr(S_{\theta^n(\phi^{(1)})}) \times  \Irr(S_{\theta(\phi^{(2)})}) $ is specified as follows;

\begin{itemize}
\item When $\phi_0=\chi^{-3}$,
$$\begin{cases}\eta_{n}^{\blacklozenge}(a_1)=\epsilon(\frac{1}{2},\phi^{(1)} \otimes \chi^{-3},\psi_2^{E}) \\ \eta_{n}^{\blacklozenge}(b_1)=\epsilon(\frac{1}{2},\phi^{(1)} \otimes \chi^{-3},\psi_2^{E}).  \end{cases}$$

\item When $\phi_0 \ne \chi^{-3}$,
$$\begin{cases}\eta_{n}^{\blacklozenge}(a_1)= \epsilon(\frac{1}{2}, \phi^{(1)} \otimes \phi_0,\psi_2^{E}) \\ \eta_{n}^{\blacklozenge}(b_1)=\epsilon(\frac{1}{2}, \phi^{(1)} \otimes \phi_0 ,\psi_2^{E})  \\ \eta_{n}^{\blacklozenge}(c_1)=\epsilon(\frac{1}{2},\phi^{(1)} \otimes \chi^{-3},\psi_2^{E}).\end{cases}$$ 
\end{itemize}
\end{thm}

\begin{proof}We first prove the existence of some $\epsilon \in \{\pm1\}$ and $\pi^{n}(\eta)\in \Pi^{R,\ \epsilon}_{\theta^{n}(\phi^{(1)},\phi^{(2)})}$ such that $$\Hom_{\Delta \U(W_3^{\epsilon})}(\pi^{n}(\eta), \omega_{\psi,W_3^{\epsilon}}) \ne 0.$$
By $(B)_1$, there is a unique $\epsilon'\in \{\pm1\}$ and a unique pair of components characters $$ (\eta_2 ,\eta_1) \in \Irr^{\epsilon'}({S}_{(\phi^{(2)})^{\vee}}) \times \Irr^{\epsilon'}({S}_{(\phi^{(1)})^{\vee}})$$ such that $$\Hom_{\Delta U(V_1^{\epsilon'})}(\pi(\eta_2)\otimes \pi(\eta_1),\mathbb{C})\ne 0.$$
Moreover, $\eta_2(b_1)=\epsilon((\phi^{(1)})^{\vee} \otimes (\phi_0)^{\vee} , \psi_{-2}^{E})=\epsilon(\frac{1}{2},\phi^{(1)} \otimes \phi_0 , \psi_{2}^{E})$ and $$\epsilon'=\eta_1(a_1)=\epsilon(\frac{1}{2},(\phi^{(1)})^{\vee} \otimes (\phi^{(2)})^{\vee}, \psi^E_{-2})=\epsilon(\frac{1}{2},\phi^{(1)} \otimes \phi^{(2)}, \psi_2^{E}).$$
By Theorem \ref{Tae} (ii) \text{and} (iv), $\Theta_{\psi,W_3^{\epsilon},V_2^{\epsilon'}}(\pi(\eta_2))$  is a nonzero  irreducible representation of $U(W_3^{\epsilon})$ for some $\epsilon \in \{\pm1\}$ and by Theorem 4.1 in \cite{ag}, $\Theta_{\psi,V_2^{\epsilon'},W_3^{\epsilon}}\big(\Theta_{\psi,W_3^{\epsilon},V_2^{\epsilon'}}(\pi(\eta_2))\big)$ is nonzero. So by our assumption, it is irreducible and we have $\Theta_{\psi,V_2^{\epsilon'},W_3^{\epsilon}}\big(\Theta_{\psi,W_3^{\epsilon},V_2^{\epsilon'}}(\pi(\eta_2))\big)=\pi(\eta_2)$.

Since $$\Hom_{\Delta U(V_1^{\epsilon'})}(\pi(\eta_2), \pi^{\vee}(\eta_1))\ne 0,$$  by the see-saw identity and Remark \ref{con}, we have $$\Hom_{\Delta U(W_3^{\epsilon})}\big(\Theta_{\psi,W_3^{\epsilon},V_1^{\epsilon'}}(\pi^{\vee}(\eta_1)) \otimes \omega^{\vee}_{\psi,W_3^{\epsilon}}, \Theta_{\psi,W_3^{\epsilon},V_2^{\epsilon'}}(\pi(\eta_2))\big)\ne 0$$ and since
$\Theta_{\psi,W_3^{\epsilon},V_2^{\epsilon'}}(\pi(\eta_2))$ and $\omega^{\vee}_{\psi,W_3^{\epsilon}}$ are admissible, we have $$\Hom_{\Delta U(W_3^{\epsilon})}\big(\Theta_{\psi,W_3^{\epsilon},V_1^{\epsilon'}}(\pi^{\vee}(\eta_1)) \otimes \Theta^{\vee}_{\psi,W_3^{\epsilon},V_2^{\epsilon'}}(\pi(\eta_2)) , \omega_{\psi,W_3^{\epsilon}} \big)\ne 0.$$
Write $\Theta_{\psi,W_3^{\epsilon},V_2^{\epsilon'}}(\pi(\eta_2))=\pi(\eta_3)$ for some $\eta_3 \in \prod_{\theta((\phi^{(2)})^{\vee})}$.\\ If $\phi_0=\chi^{-3}$, then $\eta_3(z_{\theta((\phi^{(2)})^{\vee})})=\eta_3(3\cdot b_1)=\eta_3(b_1)$ and if $\phi_0 \ne \chi^{-3}$, then $\eta_3(z_{\theta((\phi^{(2)})^{\vee})})=\eta_3(b_1)$.\\ Thus in both cases, we have $$\epsilon=\eta_3(z_{\theta((\phi^{(2)})^{\vee})})=\eta_3(b_1)=\eta_2(b_1)=\epsilon(\frac{1}{2},\phi^{(1)} \otimes \phi_0, \psi_{2}^{E}).$$
Since $$\epsilon'=\epsilon(\frac{1}{2},\phi^{(1)} \otimes \phi^{(2)} , \psi_{2}^{E})=\epsilon(\frac{1}{2},\phi^{(1)} \otimes \phi_0 , \psi_{2}^{E}) \cdot \epsilon(\frac{1}{2},\phi^{(1)} \otimes \chi^{-3} , \psi_{2}^{E}),$$ we have $\epsilon \cdot \epsilon'=\epsilon(\frac{1}{2},\phi^{(1)} \otimes \chi^{-3} , \psi_{2}^{E})$ and so by Theorem \ref{Te} (iii) and Theorem \ref{Tae} (i), the $L$-parameter of $ \Theta_{\psi,W_3^{\epsilon},V_1^{\epsilon'}}(\pi^{\vee}(\eta_1)) \otimes \Theta^{\vee}_{\psi,W_3^{\epsilon},V_2^{\epsilon'}}(\pi(\eta_2)) $ is $\theta^{n}(\phi^{(1)},\phi^{(2)})$.

Furthermore, by applying Theorem \ref{Te} (v), Theorem \ref{Tae} (iii), we see that the associated component character of $ \Theta_{\psi,W_3^{\epsilon},V_1^{\epsilon'}}(\pi^{\vee}(\eta_1)) \otimes \Theta^{\vee}_{\psi,W_3^{\epsilon},V_2^{\epsilon'}}(\pi(\eta_2)) $ is exactly $\eta_{n}^{\blacklozenge}$ in each cases and it proves the existence part. The proof of the uniqueness part is essentially same as the one in Theorem \ref{thm2}.
\end{proof}

\begin{rem}It is remarkable that for the supercuspidal $L$-parameter $\theta^{s}(\phi^{(1)},\phi^{(2)}) $ with $\theta(\phi^{(2)})$ as above, the recipe, which is sugegsted in (FJ)$_3$, does not occur from the theta lift from $U(V_1^{\pm})$ and $U(V_2^{\pm})$. This is quite similar with the Proposition 4.6 in \cite{Haan}, which concerns the non-generic aspect of \emph{Bessel} case of the GGP conjecture. 

\end{rem}

\end{document}